\documentclass[10pt]{article}
\usepackage{amsmath, amsfonts, amsthm, breqn}

\newcommand{\N}{\mathbb{N}}

\newcommand{\G}{\mathcal{G}}
\renewcommand{\H}{\mathcal{H}}
\renewcommand{\sp}{\textrm{\textvisiblespace\ }}
\newcommand{\s}{\sp&}
\newcommand{\bott}{$n$& \tiny0& \tiny1& \tiny2& \tiny3& \tiny4& \tiny5&\tiny6&\tiny7&\tiny8&\tiny9& \tiny10& \tiny11& \tiny12& \tiny13&\tiny14&\tiny15&\tiny16&\tiny17&\tiny18&\tiny19&\tiny20&\tiny21}
\newcommand{\mex}{\textrm{mex}}

\newtheorem{thm}{Theorem}
\newtheorem{defn}[thm]{Definition}
\newtheorem{lemma}[thm]{Lemma}
\newtheorem{conj}[thm]{Conjecture}
\newtheorem{alg}[thm]{Algorithm}

\title{Subtraction games with FES sets of size 3}
\author{Danny Sleator and Marla Slusky}

\begin{document}

\maketitle

\begin{abstract}
This paper extends the work done by Angela Siegel \cite{S} on subtraction games in which
 the subtraction set is $\N\setminus X$ for some finite set $X$. Siegel proves that for any
 finite set $X$, the $\G$-sequence is ultimately arithmetic periodic, and that if $|X|=1$ 
or $2$, then it is purely arithmetic periodic. This note proves that if $|X|=3$ then the 
$\G$-sequence is purely arithmetic periodic. It is known that for $|X|\geq 4$ the sequence 
is not always purely arithmetic periodic.
\end{abstract}

\section{Introduction}

The {\em subtraction game} for a set $S\subset \N$ is played on a heap of counters. A move is to choose a number from $S$ 
and remove that many counters from the heap. The set $S$ is called the {\em subtraction set}. We concern ourselves here with {\em normal play}, that is, the last player able to make a move wins. 

An {\em all-but subtraction game} is a subtraction game where the subtraction set consists of all but finitely many 
of the natural numbers. Formally, $S=\N \setminus X$ for some finite set $X\subset \N$. This $X$ is called the {\em finite excluded subtraction set} or FES set. 

As with many games, all-but subtraction games are not interesting when played in isolation; they are more interesting as part of a disjunctive sum. We are therefore interested in the nimbers of the pile sizes so that we can use Sprague-Grundy Theory \cite{LIP} to analyze the positions.

We use $\G(n)$ to mean the nimber of a pile of size $n$, that is, $\G(n) = \mex\{\G(n-s) | s\in S, s\leq n\}$. We will also call this function the {\em nim sequence} of the game.

A nim sequence is {\em periodic} (or {\em ultimately periodic}) if there exist $n_0, p\in \N$ such 
that for all $n \geq n_0$ we have $\G(n+p) = \G(n)$. If $n_0=0$ we say the nim sequence is {\em purely periodic}. It is known that if the subtraction set $S$ is finite, then the nim sequence is periodic.

A nim sequence is {\em arithmetic periodic} (or {\em ultimately arithmetic periodic}) if there exist $n_0, p, s\in \N$ such 
that for all $n\geq n_0$ we have $\G(n+p)=\G(n)+s$. We call $n_0$ the {\em preperiod}, $p$ the {\em period}, and $s$ the {\em saltus}. If 
$n_0=0$ we say the nim sequence is {\em purely arithmetic periodic}. Siegel proved \cite{S} that the nim sequence of any
 all-but subtraction game is arithmetic periodic.

Siegel also showed that in an all-but subtraction game, if the size of the FES set is 1 or 2, then the nim sequence is purely arithmetic periodic. We show in this paper that the same is true if the FES set has size 3.

\section{Background}

In calculating the nim sequence, it is natural to repeatedly apply the definition of the $\G$ function. However, for the proofs in this paper, we require a different algorithm which we will call the FES algorithm.

\begin{alg}\label{fes} (FES Algorithm)\\
 For $k=0, 1, 2, \ldots$ find all $m$ such that $\G(m)=k$ as follows:\\
 \hspace*{.3 in} Let $n=\min\{i:\G(i) \textrm{ is unknown }\}$. $\G(n)=k$.\\
 \hspace*{.3 in} For $x\in X$ considered in increasing order,\\
 \hspace*{.6 in} If $\G(n+x)$ is unknown and for all $m<n+x$ with $\G(m)=k$ we have $(n+x)-m\in X$ \\
 \hspace*{.6 in} then $\G(n+x)=k$.
\end{alg}

\begin{thm}
 Algorithm \ref{fes} correctly calculates the nim sequence.
\end{thm}
\begin{proof}
The proof is by induction.
   The induction hypothesis is that after $k$ iterations the algorithm
   has corrected labeled all the positions that have a nimber
   in the set $\{0, 1, \ldots, k-1\}$.  We need to prove that the next iteration
   correctly computes the placement of the $k$'s.  We know that
   position $n$ must have a $k$ in it, because its value (by induction)
   is $> k-1$, but it cannot be $>k$ because there are no $k$'s before it.
   The positions where it places the remaining $k$'s are precisely
   those where you cannot reach a $k$ by making a move.  Thus
   by the same argument they must have a value of $k$.
\end{proof}
As an example, here are the first few steps of the algorithm carried out for the FES set $X=\{2, 3, 6, 8\}$\\

\begin{tabular}{ccccccccccccccccccccccc}
\\
 $\G(n)$&\s\s\s\s\s\s\s\s\s\s\s\s\s\s\s\s\s\s\s\s\s\sp\\
\bott\\
\\
 $\G(n)$& 0& \s 0& \s\s\s\s\s0&\s\s\s\s\s\s\s\s\s\s\s\s\sp\\
 \bott\\
\\
 $\G(n)$& 0&  1 & 0&  1&\s\s\s\s 0& 1&\s\s\s\s\s\s\s\s\s\s\s\sp\\
 \bott\\
\\
 $\G(n)$& 0&  1 & 0&  1& 2&\s 2&\s 0& 1&\s\s 2&\s\s\s\s\s\s\s\s\sp\\
 \bott\\
\\
 $\G(n)$& 0&  1 & 0&  1& 2& 3& 2& 3& 0& 1&\s\s 2& 3&\s\s\s\s\s\s\s\sp\\
 \bott\\
\\
$\G(n)$& 0&  1 & 0&  1& 2& 3& 2& 3& 0& 1& 4&\s    2& 3&\s   \s 4&\s    4&\s\s\sp\\
 \bott\\
\\
$\G(n)$& 0&  1 & 0&  1& 2& 3& 2& 3& 0& 1& 4& 5& 2& 3& 5&\s 4& 5& 4&\s\s\sp\\
 \bott\\
&&&&&&&&&&$\vdots$
\\\\
\end{tabular}

Note that after the first $k-1$ steps of this algorithm have been carried out, if we want to find which 
piles have nimber $k$, we don't need to know which piles have which nimber, only which piles have nimbers 
less than $k$. For this reason, after each step, we want to just think about $n$-values as either having a nimber 
or not yet having a nimber. Define the function 
$\H_k(n) : \N \to\{\textrm{ * }, \textrm{ \sp\ }\}$ by $\H_k(n)=$ * iff $\G(n) < k$ and $\H_k(n)=$\nolinebreak \sp\   iff $\G(n) \geq k$. For example, compare $\H_3$ to our state of knowledge after the third iteration of the FES algorithm.

\begin{tabular}{ccccccccccccccccccccccc}
\\
 $\H_3(n)$&  *&   * &  *&   *&  *&\s  *&\s  *&  *&\s\s  *&\s\s\s\s\s\s\s\s\sp\\
 \bott\\
\\
\end{tabular}

Note that the beginning ``chunk'' is all *s and the end ``chunk'' is all blanks. We are interested in the middle ``chunk,'' the short interval in which $\H_k$ takes on values of both   * and \sp.

\begin{lemma}
For a fixed $k$, let $n = \min\{i:\H_k(i)=\textrm{\sp }\}$. Then for $m \geq n+\max (X)$, we have $\H_k(m)=$\sp. 
 \end{lemma}
\begin{proof}
 Let $m \geq n+\max (X)$, let $\ell<k$ be a nimber, and let $p$ be the smallest pile such that $\G(p)=\ell$. Then $p < n$. 
 Then since $m \geq n + \max (X)  > p+\max (X)$, we can move from a pile of size $m$ to a pile of size $p$. 
	This tells us that 
 $\G(m) \neq \ell$. Thus $\G(m) \geq k$ so $\H_k(m)=$\sp.
\end{proof}

\begin{defn}
 The {\em boundary pattern} of $\H_k$ is the sequence \\$\H_k(n), \H_k(n+1), \H_k(n+2), \ldots, \H_k(n+\max (X)-1)$ where
 $n=\min\{i:\H_k(i)=\textrm{\sp }\}$.
\end{defn}

These boundary patterns characterize the $\H_k$;s in that knowing the boundary pattern of $\H_{k-1}$ is sufficient information to find the boundary pattern of $\H_k$. Note that the sequence of boundary patterns is ultimately 
periodic if and only if the  nim sequence is ultimately arithmetic periodic, and that the sequence of boundary patterns is purely periodic if and only if the nim sequence is purely arithmetic periodic. 


We can use this observation to reprove Siegel's Theorem 8.

\begin{thm}
 Given a finite set $X$, the nim sequence for the all-but subtraction game on with FES set $X$ is eventually arithmetic periodic.
\end{thm}
\begin{proof}
 Consider directed graph where the vertex set is all possible boundary patterns, and each vertex has
exactly one 
outgoing edge pointing to the next boundary pattern in the sequence. Since this graph is finite, 
(it has at most $2^{\max( X) -1}$ vertices) following any path eventually leads into a cycle, so the  nim sequence 
 is eventually ultimately periodic.
\end{proof}

We will need the following crucial lemma from Siegel \cite{S}.
\begin{lemma}\label{15}
 If $X = \{a, b, a+b\}$ where $b>a$ and $b\neq 2a$, then the nim sequence of the associated all-but subtraction game contains every non-negative integer $k$ exactly three times with either $k=\G(n)=\G(n+a)=\G(n+a+b)$ or, failing that, $k=\G(n)=\G(n+b)=\G(n+a+b)$.
\end{lemma}
\begin{proof}
 The proof is by induction on $k$. Since every previous nimber appears only three times, 
$k$ will appear somewhere. Set $n = \min\{i:\G(i)=k\}$. Following the FES algorithm, if we set $\G(n+a)=k$, then 
since $(n+b)-(n+a)\notin X$, $\G(n+b)\neq k$. If $\G(n+a) < k$ then the FES algorithm sets $\G(n+b)=k$ 
unless $\G(n+b)<k$. If this is the case, take $\G(n+b)=\ell$ and $m=\min\{i:\G(i)=\ell\}$. Then $m<n$ so 
$n+b=m+a+b$. However this implies that $m+a=n$, so the FES algorithm would have set $\G(n)=\ell$. 
This is a contradiction, so $\G(n+b)=k$.\\

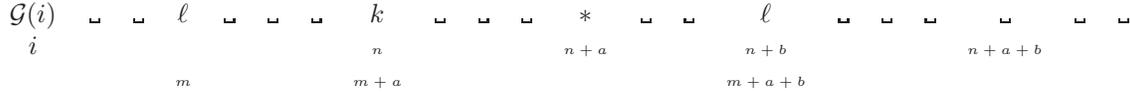
\begin{tabular}{ccccccccccccccccccccccc}
 $\G(i)$      & \s   \s    $\ell$     &      \s  \s  \s   $k$&  \s   \s  \s     $*$     &   \s  \s     $\ell$&\s  \s  \s  \s        \s  \sp \\
 $i$          &    &    &             &   &    &   &\tiny $n$&    &    &   & \tiny $n+a$&  & &\tiny $n+b$&   &   &   &\tiny $n+a+b$&   &    \\
              &    &    & \tiny $m$           &   &    &   &\tiny $m+a$ &    &    &   &     &  & &\tiny $m+a+b$&   &   &   &           &   &    \\
\end{tabular}

The last step of the FES algoithm for finding $i$ with $\G(i)=k$ considers $\G(n+a+b)$. However, $n+a+b$ is stricly greater than any previously considered pile sizes, so $\G(n+a+b)$ is so far unknown. Thus the FES algorithm sets $\G(n+a+b)=k$. 

\end{proof}

\section{FES sets of size 3}
Let $\G_X$ be the nim sequence for the all-but subtraction game with FES set $X$. Siegel shows that if $|X|=1$ or $|X|=2$ then $\G_X$ is purely arithmetic periodic. She makes some conjectures about $|X|=3$, but leaves it largely open. In this section we show that if $|X|=3$ then $\G_X$ is purely arithmetic periodic.

\begin{lemma}
\label{a+b}
 If $X=\{a,b,a+b\}$ then $\G_X$ is purely periodic.
\end{lemma}
\begin{proof}
 Consider again the directed graph of boundary patterns. It is sufficient to prove that every vertex 
has an in-degree of exactly 1 since 
this will imply that the graph is just a collection of one
or more disjoint cycles.  Thus starting from the initial boundary pattern, if we follow the edges, 
we must eventually return to the initial boundary pattern. 

Consider $\H_{k}$. We will show that we can uniquely reconstruct $\H_{k-1}$ from $\H_k$. We know that the first occurences of the nimbers occur in order. That is, $\min\{i : \G(i)=k-1\} \leq \min\{i:\G(i)=k\}$. 
Therefore lemma \ref{15} implies that the third occurences also occur in order. 

From here we know that $\max\{i:\H_k(i)=*\}$ is the third occurence of $k-1$. In other words, if 
$n+a+b = \max\{i:\H_k(i)=*\}$ then $k-1=\G(n+a+b)=\G(n)$.
Now the only question is which of $n+a$ and $n+b$ has a $\G$-value of $k-1$? If only one of $\H_{k}(n+a)$ 
and $\H_{k}(n+b)$ is *, then that is the one with $\G$-value $k-1$. The only possible confusion comes when 
$\H_k(n+a)=\H_k(n+b)=*$. In this situation $\G(n+b)=k-1$.

Assume for contradiction that $\G(n+a)=k-1$. Then $\G(n+b)=\ell < k-1$. Since $\ell <k-1$, we know that $n+b$ must be the third occurence of $\ell$. That is, $\ell = \G(m)=\G(m+a+b)$ and $m+a+b=n+b$.

\begin{tabular}{ccccccccccccccccccccccc}
\\
 $\G(i)$      & \s   \s    $\ell$     &      \s  \s  \s   $k-1$& \s   \s  \s     $k-1$    &   \s  \s     $\ell$&\s  \s  \s  $k-1$&        \s  \sp \\
 $i$          &    &    &              &   &    &   &\tiny $n$&    &    &   & \tiny $n+a$&  & &\tiny $n+b$&   &   &   &\tiny $n+a+b$&   &    \\
              &    &    & \tiny $m$           &   &    &   &\tiny $m+a$ &    &    &   &     &  & &\tiny $m+a+b$&   &   &   &           &   &    \\

\\
\end{tabular}
\\
However, since $\ell < k-1$, this means $\H_{\ell}(m+a) = \sp$. This is a contradiction, because lemma \ref{15} 
implies that $\G(m+a)=\ell$.

Now we know that given $\H_k$ we can find which values of $m$ have $\G(m)=k-1$ and from this we can construct $\H_{k-1}$. This means that from a given boundary pattern, we can uniquely determine the boundary pattern that proceeds it, and so our proof is complete.
\end{proof}

\begin{lemma}
 If $b>a$ and $b\neq 2a$ then $\G_{\{a, b, 2a\}} = \G_{\{a,2a\}}$. 
\end{lemma}
\begin{proof}
We show by induction that the only boundary patterns that arise have the form \\
 \sp\sp\sp\sp***\sp\sp\sp\sp***. That is, $i$ copies of \sp followed by $a-i$ copies of *, then another $i$ copies of \sp and another $a-i$ copies of *.

Our first boundary pattern is blank and so it has the form desired form with $i=a$. From such a 
boundary pattern we run the FES algorithm. Take $n= \min\{i:\G(i) \textrm{ is unknown }\}$. Then $\G(n)=k$. Then $\G(n+a)=k$. Then $n+b$ has $n+a$ as an option so $\G(n+b)\neq k$. Then $\G(n+2a)=k$. The new boundary pattern is of the same from, with $i$ decreased by 1 (mod $a$).

Since the FES algorithm does never sets $\G(n+b)=k$, it runs the same as it would if $X=\{a, 2a\}$.
\end{proof}

\begin{lemma}
 If $b\neq 2a$ and $X=\{a, b, 2b\}$ then $\G_X = \G_{\{a\}}$.
\end{lemma}

\begin{proof}
 The argument is similar to the previous one. We show by induction the only boundary patterns 
that arise have the form  \sp\sp\sp\sp***. That is, $i$ copies of \sp followed by $a-i$ copies of *.

Our first boundary pattern is blank and so it has the form desired form with $i=a$. From such a boundary patterns we run the FES algorithm. Take $n= \min\{i:\G(i) \textrm{ is unknown }\}$. Then $\G(n)=k$. Then $\G(n+a)=k$. Then $n+b$ and $n+2b$ have $n+a$ as an option so $\G(n+b)\neq k$ and $\G(n+2b)\neq k$. The new boundary pattern is of the same from, with $i$ decreased by 1 (mod $a$).

Since the FES algorithm does never sets $\G(n+b)=k$ nor $\G(n+2b)=k$, it runs the same as it would if $X=\{a\}$.

\end{proof}

\begin{lemma}
 Let $X = \{a,b,c\}$ where $c\neq a+b$, $c\neq 2a$ and $c \neq 2b$. Then $\G_{\{a,b,c\}} = \G_{\{a,b\}}$.
\end{lemma}

\begin{proof}
 Assume not, and let $n$ be the first pile on which they differ. That is, 
$k=\G_{\{a,b,c\}}(n) \neq \G_{\{a,b\}}(n)=\ell$ and 
$\G_{\{a,b,c\}}(m) = \G_{\{a,b\}}(m)$ for all $m<n$. The mex sets for $\G_{\{a,b,c\}}(n)$ and $\G_{\{a,b\}}(n)$ only differ by one element, $\G(n-c)$. Since that one element is causing the first difference, we must have $\G(n-c)=k$.
 
\begin{tabular}{ccccccccccccccccccccccc}
\\
 $\G_{\{a,b\}}(i)$  & \s  \s  $k$&      \s  \s  \s  \s  \s  \s  \s  \s  \s  \s  \s  \s  \s  $\ell$&        \s  \sp \\
 $i$                &    &  &\tiny $n-c$&   &    &   &\tiny $n-c+a$&   &   &   &   &\tiny $n-c+b$&   &   &   &   &\tiny $n$&   &    \\
\\
 $\G_{X}(i)$& \s  \s  $k$ &     \s  \s  \s  \s  \s  \s  \s  \s  \s  \s  \s  \s  \s  $k$&      \s  \sp \\
 $i$                &    &  &\tiny $n-c$&   &   &    &\tiny $n-c+a$&   &   &   &   &\tiny $n-c+b$&   &   &   &   &\tiny $n$&   &    \\
\\
\end{tabular}

If $\G( (n-c) + a) = k$ then since $\G_X(n)=k$, we must have that $n-c+a$ is not an option of $n$. This means $n-a=n-c+a$ or $n-b = n-c+a$. However, this implies $c=2a$ or $c=a+b$. 

If $\G( (n-c) + a) \neq k$, then in $\{a,b\}$ the options of $n-c+b$ are a superset of the options of $n-c$, so $\G_{\{a,b\}}(n-c+b)\geq k$. For $n-c<m<n-c+b$, $n-c$ is an option of $m$, so $\G_{\{a,b\}}(m)\neq k$. 
Since $n-c$ is not an option of $n-c+b$ we have $\G_{\{a,b\}}(n-c+b)=k$.

\begin{tabular}{ccccccccccccccccccccccc}
\\
 $\G_{\{a,b\}}(i)$  & \s  \s  $k$&      \s  \s  \s  \s  \s  \s  \s  \s  $k$&  \s  \s  \s  \s  $\ell$&        \s  \sp \\
 $i$                &    &  &\tiny $n-c$&   &    &   &\tiny $n-c+a$&   &   &   &   &\tiny $n-c+b$&   &   &   &   &\tiny $n$&   &    \\
\\
 $\G_{X}(i)$        & \s  \s  $k$ &     \s  \s  \s  \s  \s  \s  \s  \s  $k$&  \s  \s  \s  \s  $k$   &      \s  \sp \\
 $i$                &    &  &\tiny $n-c$&   &   &    &\tiny $n-c+a$&   &   &   &   &\tiny $n-c+b$&   &   &   &   &\tiny $n$&   &    \\
\\
\end{tabular}

However now we have that $n-c+b$ is not an option of $n$, so $n-a=n-c+b$ or $n-b=n-c+b$. This implies that $c=a+b$ or $c=2b$.

\end{proof}

\begin{thm}
 If $|X|=3$ then $\G_X$ is purely arithmetic periodic.
\end{thm}
\begin{proof}
 The proof breaks down into three cases, which are taken care of in the previous three lemmas.
\end{proof}

\section{Conjectures \& Future Work}
The proof that $\G_{\{a,b,a+b\}}$ is purely periodic gives no insight into the length of the period. 
However there 
seem to be some obvious patterns in the case where $b>3a$. 
Lemma \ref{15} implies that the period is three times the saltus, so we only need to consider the saltus. It is known that the period for $\{na, nb, n(a+b)\}$ is $n$ times the period for $\{a, b, a+b\}$ so we only need to consider FES sets where $a$ and $b$ are relatively prime. 

\begin{conj}
 Let $a$ and $b$ be such that $b>3a$ and $\gcd(a, b)=1$ and let $p$ be the period of the nim
 sequence for the FES set $\{a, b, a+b\}$. If there exists an $m$ that is a multiple of $2a$ with $b<m<a+b$ 
$p=3am$. If no such $m$ exists, then there is some other $n$ with $b<n<a+b$ such that $p=3an$.
\end{conj}

The following pages contain the data from which this conjecture was drawn.

\begin{tabular}{cccc}
\begin{tabular}{|c|c|}
\hline 
FES set & saltus\\
\hline\hline
\{1, 3, 4\}    &      1 * 4 \\
\hline
\{1, 4, 5\}    &      1 * 4 \\
\hline
\{1, 5, 6\}    &      1 * 6 \\
\hline
\{1, 6, 7\}    &      1 * 6 \\
\hline
\{1, 7, 8\}    &      1 * 8 \\
\hline
\{1, 8, 9\}    &      1 * 8 \\
\hline
\{1, 9, 10\}    &      1 * 10 \\
\hline
\{1, 10, 11\}    &      1 * 10 \\
\hline
\{1, 11, 12\}    &      1 * 12 \\
\hline
\{1, 12, 13\}    &      1 * 12 \\
\hline
\{1, 13, 14\}    &      1 * 14 \\
\hline
\{1, 14, 15\}    &      1 * 14 \\
\hline
\{1, 15, 16\}    &      1 * 16 \\
\hline
\{1, 16, 17\}    &      1 * 16 \\
\hline
\{1, 17, 18\}    &      1 * 18 \\
\hline
\{1, 18, 19\}    &      1 * 18 \\
\hline
\{1, 19, 20\}    &      1 * 20 \\
\hline
\{1, 20, 21\}    &      1 * 20 \\
\hline
\{1, 21, 22\}    &      1 * 22 \\
\hline
\{1, 22, 23\}    &      1 * 22 \\
\hline
\{1, 23, 24\}    &      1 * 24 \\
\hline
\{1, 24, 25\}    &      1 * 24 \\
\hline
\{1, 25, 26\}    &      1 * 26 \\
\hline
\{1, 26, 27\}    &      1 * 26 \\
\hline
\{1, 27, 28\}    &      1 * 28 \\
\hline
\{1, 28, 29\}    &      1 * 28 \\
\hline
\{1, 29, 30\}    &      1 * 30 \\
\hline
\{1, 30, 31\}    &      1 * 30 \\
\hline
\{1, 31, 32\}    &      1 * 32 \\
\hline
\{1, 32, 33\}    &      1 * 32 \\
\hline
\{1, 33, 34\}    &      1 * 34 \\
\hline
\{1, 34, 35\}    &      1 * 34 \\
\hline
\{1, 35, 36\}    &      1 * 36 \\
\hline
\{1, 36, 37\}    &      1 * 36 \\
\hline
\{1, 37, 38\}    &      1 * 38 \\
\hline
\{1, 38, 39\}    &      1 * 38 \\
\hline
\{1, 39, 40\}    &      1 * 40 \\
\hline
\{1, 40, 41\}    &      1 * 40 \\
\hline
\{1, 41, 42\}    &      1 * 42 \\
\hline
\{1, 42, 43\}    &      1 * 42 \\
\hline

\end{tabular}
&
\begin{tabular}{|c|c|}
\hline 
FES set & saltus\\
\hline\hline
\{2, 7, 9\}    &      2 * 8 \\
\hline
\{2, 9, 11\}    &      2 * 10 \\
\hline
\{2, 11, 13\}    &      2 * 12 \\
\hline
\{2, 13, 15\}    &      2 * 14 \\
\hline
\{2, 15, 17\}    &      2 * 16 \\
\hline
\{2, 17, 19\}    &      2 * 18 \\
\hline
\{2, 19, 21\}    &      2 * 20 \\
\hline
\{2, 21, 23\}    &      2 * 22 \\
\hline
\{2, 23, 25\}    &      2 * 24 \\
\hline
\{2, 25, 27\}    &      2 * 26 \\
\hline
\{2, 27, 29\}    &      2 * 28 \\
\hline
\{2, 29, 31\}    &      2 * 30 \\
\hline
\{2, 31, 33\}    &      2 * 32 \\
\hline
\{2, 33, 35\}    &      2 * 34 \\
\hline
\{2, 35, 37\}    &      2 * 36 \\
\hline
\{2, 37, 39\}    &      2 * 38 \\
\hline
\{2, 39, 41\}    &      2 * 40 \\
\hline
\{2, 41, 43\}    &      2 * 42 \\
\hline
\{2, 43, 45\}    &      2 * 44 \\
\hline
\{2, 45, 47\}    &      2 * 46 \\
\hline
\{2, 47, 49\}    &      2 * 48 \\
\hline
\{2, 49, 51\}    &      2 * 50 \\
\hline
\{2, 51, 53\}    &      2 * 52 \\
\hline
\{2, 53, 55\}    &      2 * 54 \\
\hline
\{2, 55, 57\}    &      2 * 56 \\
\hline
\{2, 57, 59\}    &      2 * 58 \\
\hline
\{2, 59, 61\}    &      2 * 60 \\
\hline
\{2, 61, 63\}    &      2 * 62 \\
\hline
\{2, 63, 65\}    &      2 * 64 \\
\hline
\{2, 65, 67\}    &      2 * 66 \\
\hline
\{2, 67, 69\}    &      2 * 68 \\
\hline
\{2, 69, 71\}    &      2 * 70 \\
\hline
\{2, 71, 73\}    &      2 * 72 \\
\hline
\{2, 73, 75\}    &      2 * 74 \\
\hline
\{2, 75, 77\}    &      2 * 76 \\
\hline
\{2, 77, 79\}    &      2 * 78 \\
\hline
\{2, 79, 81\}    &      2 * 80 \\
\hline
\{2, 81, 83\}    &      2 * 82 \\
\hline
\{2, 83, 85\}    &      2 * 84 \\
\hline
\{2, 85, 87\}    &      2 * 86 \\
\hline

\end{tabular}
&
\begin{tabular}{|c|c|}
\hline 
FES set & saltus\\
\hline\hline
\{3, 10, 13\}    &      3 * 12 \\
\hline
\{3, 11, 14\}    &      3 * 12 \\
\hline
\{3, 13, 16\}    &      3 * 14 \\
\hline
\{3, 14, 17\}    &      3 * 16 \\
\hline
\{3, 16, 19\}    &      3 * 18 \\
\hline
\{3, 17, 20\}    &      3 * 18 \\
\hline
\{3, 19, 22\}    &      3 * 20 \\
\hline
\{3, 20, 23\}    &      3 * 22 \\
\hline
\{3, 22, 25\}    &      3 * 24 \\
\hline
\{3, 23, 26\}    &      3 * 24 \\
\hline
\{3, 25, 28\}    &      3 * 26 \\
\hline
\{3, 26, 29\}    &      3 * 28 \\
\hline
\{3, 28, 31\}    &      3 * 30 \\
\hline
\{3, 29, 32\}    &      3 * 30 \\
\hline
\{3, 31, 34\}    &      3 * 32 \\
\hline
\{3, 32, 35\}    &      3 * 34 \\
\hline
\{3, 34, 37\}    &      3 * 36 \\
\hline
\{3, 35, 38\}    &      3 * 36 \\
\hline
\{3, 37, 40\}    &      3 * 38 \\
\hline
\{3, 38, 41\}    &      3 * 40 \\
\hline
\{3, 40, 43\}    &      3 * 42 \\
\hline
\{3, 41, 44\}    &      3 * 42 \\
\hline
\{3, 43, 46\}    &      3 * 44 \\
\hline
\{3, 44, 47\}    &      3 * 46 \\
\hline
\{3, 46, 49\}    &      3 * 48 \\
\hline
\{3, 47, 50\}    &      3 * 48 \\
\hline
\{3, 49, 52\}    &      3 * 50 \\
\hline
\{3, 50, 53\}    &      3 * 52 \\
\hline
\{3, 52, 55\}    &      3 * 54 \\
\hline
\{3, 53, 56\}    &      3 * 54 \\
\hline
\{3, 55, 58\}    &      3 * 56 \\
\hline
\{3, 56, 59\}    &      3 * 58 \\
\hline
\{3, 58, 61\}    &      3 * 60 \\
\hline
\{3, 59, 62\}    &      3 * 60 \\
\hline
\{3, 61, 64\}    &      3 * 62 \\
\hline
\{3, 62, 65\}    &      3 * 64 \\
\hline
\{3, 64, 67\}    &      3 * 66 \\
\hline
\{3, 65, 68\}    &      3 * 66 \\
\hline
\{3, 67, 70\}    &      3 * 68 \\
\hline
\{3, 68, 71\}    &      3 * 70 \\
\hline

\end{tabular}
&
\begin{tabular}{|c|c|}
\hline 
FES set & saltus\\
\hline\hline
\{4, 13, 17\}    &      4 * 16 \\
\hline
\{4, 15, 19\}    &      4 * 16 \\
\hline
\{4, 17, 21\}    &      4 * 18 \\
\hline
\{4, 19, 23\}    &      4 * 22 \\
\hline
\{4, 21, 25\}    &      4 * 24 \\
\hline
\{4, 23, 27\}    &      4 * 24 \\
\hline
\{4, 25, 29\}    &      4 * 26 \\
\hline
\{4, 27, 31\}    &      4 * 30 \\
\hline
\{4, 29, 33\}    &      4 * 32 \\
\hline
\{4, 31, 35\}    &      4 * 32 \\
\hline
\{4, 33, 37\}    &      4 * 34 \\
\hline
\{4, 35, 39\}    &      4 * 38 \\
\hline
\{4, 37, 41\}    &      4 * 40 \\
\hline
\{4, 39, 43\}    &      4 * 40 \\
\hline
\{4, 41, 45\}    &      4 * 42 \\
\hline
\{4, 43, 47\}    &      4 * 46 \\
\hline
\{4, 45, 49\}    &      4 * 48 \\
\hline
\{4, 47, 51\}    &      4 * 48 \\
\hline
\{4, 49, 53\}    &      4 * 50 \\
\hline
\{4, 51, 55\}    &      4 * 54 \\
\hline
\{4, 53, 57\}    &      4 * 56 \\
\hline
\{4, 55, 59\}    &      4 * 56 \\
\hline
\{4, 57, 61\}    &      4 * 58 \\
\hline
\{4, 59, 63\}    &      4 * 62 \\
\hline
\{4, 61, 65\}    &      4 * 64 \\
\hline
\{4, 63, 67\}    &      4 * 64 \\
\hline
\{4, 65, 69\}    &      4 * 66 \\
\hline
\{4, 67, 71\}    &      4 * 70 \\
\hline
\{4, 69, 73\}    &      4 * 72 \\
\hline
\{4, 71, 75\}    &      4 * 72 \\
\hline
\{4, 73, 77\}    &      4 * 74 \\
\hline
\{4, 75, 79\}    &      4 * 78 \\
\hline
\{4, 77, 81\}    &      4 * 80 \\
\hline
\{4, 79, 83\}    &      4 * 80 \\
\hline
\{4, 81, 85\}    &      4 * 82 \\
\hline
\{4, 83, 87\}    &      4 * 86 \\
\hline
\{4, 85, 89\}    &      4 * 88 \\
\hline
\{4, 87, 91\}    &      4 * 88 \\
\hline
\{4, 89, 93\}    &      4 * 90 \\
\hline
\{4, 91, 95\}    &      4 * 94 \\
\hline

\end{tabular}
\end{tabular}

\begin{tabular}{cccc}
\begin{tabular}{|c|c|}
\hline 
FES set & saltus\\
\hline\hline
\{5, 16, 21\}    &      5 * 20 \\
\hline
\{5, 17, 22\}    &      5 * 20 \\
\hline
\{5, 18, 23\}    &      5 * 20 \\
\hline
\{5, 19, 24\}    &      5 * 20 \\
\hline
\{5, 21, 26\}    &      5 * 22 \\
\hline
\{5, 22, 27\}    &      5 * 24 \\
\hline
\{5, 23, 28\}    &      5 * 26 \\
\hline
\{5, 24, 29\}    &      5 * 28 \\
\hline
\{5, 26, 31\}    &      5 * 30 \\
\hline
\{5, 27, 32\}    &      5 * 30 \\
\hline
\{5, 28, 33\}    &      5 * 30 \\
\hline
\{5, 29, 34\}    &      5 * 30 \\
\hline
\{5, 31, 36\}    &      5 * 32 \\
\hline
\{5, 32, 37\}    &      5 * 34 \\
\hline
\{5, 33, 38\}    &      5 * 36 \\
\hline
\{5, 34, 39\}    &      5 * 38 \\
\hline
\{5, 36, 41\}    &      5 * 40 \\
\hline
\{5, 37, 42\}    &      5 * 40 \\
\hline
\{5, 38, 43\}    &      5 * 40 \\
\hline
\{5, 39, 44\}    &      5 * 40 \\
\hline
\{5, 41, 46\}    &      5 * 42 \\
\hline
\{5, 42, 47\}    &      5 * 44 \\
\hline
\{5, 43, 48\}    &      5 * 46 \\
\hline
\{5, 44, 49\}    &      5 * 48 \\
\hline
\{5, 46, 51\}    &      5 * 50 \\
\hline
\{5, 47, 52\}    &      5 * 50 \\
\hline
\{5, 48, 53\}    &      5 * 50 \\
\hline
\{5, 49, 54\}    &      5 * 50 \\
\hline
\{5, 51, 56\}    &      5 * 52 \\
\hline
\{5, 52, 57\}    &      5 * 54 \\
\hline
\{5, 53, 58\}    &      5 * 56 \\
\hline
\{5, 54, 59\}    &      5 * 58 \\
\hline
\{5, 56, 61\}    &      5 * 60 \\
\hline
\{5, 57, 62\}    &      5 * 60 \\
\hline
\{5, 58, 63\}    &      5 * 60 \\
\hline
\{5, 59, 64\}    &      5 * 60 \\
\hline
\{5, 61, 66\}    &      5 * 62 \\
\hline
\{5, 62, 67\}    &      5 * 64 \\
\hline
\{5, 63, 68\}    &      5 * 66 \\
\hline
\{5, 64, 69\}    &      5 * 68 \\
\hline

\end{tabular}
&
\begin{tabular}{|c|c|}
\hline 
FES set & saltus\\
\hline\hline
\{6, 19, 25\}    &      6 * 24 \\
\hline
\{6, 23, 29\}    &      6 * 24 \\
\hline
\{6, 25, 31\}    &      6 * 26 \\
\hline
\{6, 29, 35\}    &      6 * 34 \\
\hline
\{6, 31, 37\}    &      6 * 36 \\
\hline
\{6, 35, 41\}    &      6 * 36 \\
\hline
\{6, 37, 43\}    &      6 * 38 \\
\hline
\{6, 41, 47\}    &      6 * 46 \\
\hline
\{6, 43, 49\}    &      6 * 48 \\
\hline
\{6, 47, 53\}    &      6 * 48 \\
\hline
\{6, 49, 55\}    &      6 * 50 \\
\hline
\{6, 53, 59\}    &      6 * 58 \\
\hline
\{6, 55, 61\}    &      6 * 60 \\
\hline
\{6, 59, 65\}    &      6 * 60 \\
\hline
\{6, 61, 67\}    &      6 * 62 \\
\hline
\{6, 65, 71\}    &      6 * 70 \\
\hline
\{6, 67, 73\}    &      6 * 72 \\
\hline
\{6, 71, 77\}    &      6 * 72 \\
\hline
\{6, 73, 79\}    &      6 * 74 \\
\hline
\{6, 77, 83\}    &      6 * 82 \\
\hline
\{6, 79, 85\}    &      6 * 84 \\
\hline
\{6, 83, 89\}    &      6 * 84 \\
\hline
\{6, 85, 91\}    &      6 * 86 \\
\hline
\{6, 89, 95\}    &      6 * 94 \\
\hline
\{6, 91, 97\}    &      6 * 96 \\
\hline
\{6, 95, 101\}    &      6 * 96 \\
\hline
\{6, 97, 103\}    &      6 * 98 \\
\hline
\{6, 101, 107\}    &      6 * 106 \\
\hline
\{6, 103, 109\}    &      6 * 108 \\
\hline
\{6, 107, 113\}    &      6 * 108 \\
\hline
\{6, 109, 115\}    &      6 * 110 \\
\hline
\{6, 113, 119\}    &      6 * 118 \\
\hline
\{6, 115, 121\}    &      6 * 120 \\
\hline
\{6, 119, 125\}    &      6 * 120 \\
\hline
\{6, 121, 127\}    &      6 * 122 \\
\hline
\{6, 125, 131\}    &      6 * 130 \\
\hline
\{6, 127, 133\}    &      6 * 132 \\
\hline
\{6, 131, 137\}    &      6 * 132 \\
\hline
\{6, 133, 139\}    &      6 * 134 \\
\hline
\{6, 137, 143\}    &      6 * 142 \\
\hline

\end{tabular}
&
\begin{tabular}{|c|c|}
\hline 
FES set & saltus\\
\hline\hline
\{7, 22, 29\}    &      7 * 28 \\
\hline
\{7, 23, 30\}    &      7 * 28 \\
\hline
\{7, 24, 31\}    &      7 * 28 \\
\hline
\{7, 25, 32\}    &      7 * 28 \\
\hline
\{7, 26, 33\}    &      7 * 28 \\
\hline
\{7, 27, 34\}    &      7 * 28 \\
\hline
\{7, 29, 36\}    &      7 * 30 \\
\hline
\{7, 30, 37\}    &      7 * 32 \\
\hline
\{7, 31, 38\}    &      7 * 34 \\
\hline
\{7, 32, 39\}    &      7 * 36 \\
\hline
\{7, 33, 40\}    &      7 * 38 \\
\hline
\{7, 34, 41\}    &      7 * 40 \\
\hline
\{7, 36, 43\}    &      7 * 42 \\
\hline
\{7, 37, 44\}    &      7 * 42 \\
\hline
\{7, 38, 45\}    &      7 * 42 \\
\hline
\{7, 39, 46\}    &      7 * 42 \\
\hline
\{7, 40, 47\}    &      7 * 42 \\
\hline
\{7, 41, 48\}    &      7 * 42 \\
\hline
\{7, 43, 50\}    &      7 * 44 \\
\hline
\{7, 44, 51\}    &      7 * 46 \\
\hline
\{7, 45, 52\}    &      7 * 48 \\
\hline
\{7, 46, 53\}    &      7 * 50 \\
\hline
\{7, 47, 54\}    &      7 * 52 \\
\hline
\{7, 48, 55\}    &      7 * 54 \\
\hline
\{7, 50, 57\}    &      7 * 56 \\
\hline
\{7, 51, 58\}    &      7 * 56 \\
\hline
\{7, 52, 59\}    &      7 * 56 \\
\hline
\{7, 53, 60\}    &      7 * 56 \\
\hline
\{7, 54, 61\}    &      7 * 56 \\
\hline
\{7, 55, 62\}    &      7 * 56 \\
\hline
\{7, 57, 64\}    &      7 * 58 \\
\hline
\{7, 58, 65\}    &      7 * 60 \\
\hline
\{7, 59, 66\}    &      7 * 62 \\
\hline
\{7, 60, 67\}    &      7 * 64 \\
\hline
\{7, 61, 68\}    &      7 * 66 \\
\hline
\{7, 62, 69\}    &      7 * 68 \\
\hline
\{7, 64, 71\}    &      7 * 70 \\
\hline
\{7, 65, 72\}    &      7 * 70 \\
\hline
\{7, 66, 73\}    &      7 * 70 \\
\hline
\{7, 67, 74\}    &      7 * 70 \\
\hline

\end{tabular}
&
\begin{tabular}{|c|c|}
\hline 
FES set & saltus\\
\hline\hline
\{8, 25, 33\}    &      8 * 32 \\
\hline
\{8, 27, 35\}    &      8 * 32 \\
\hline
\{8, 29, 37\}    &      8 * 32 \\
\hline
\{8, 31, 39\}    &      8 * 32 \\
\hline
\{8, 33, 41\}    &      8 * 34 \\
\hline
\{8, 35, 43\}    &      8 * 38 \\
\hline
\{8, 37, 45\}    &      8 * 42 \\
\hline
\{8, 39, 47\}    &      8 * 46 \\
\hline
\{8, 41, 49\}    &      8 * 48 \\
\hline
\{8, 43, 51\}    &      8 * 48 \\
\hline
\{8, 45, 53\}    &      8 * 48 \\
\hline
\{8, 47, 55\}    &      8 * 48 \\
\hline
\{8, 49, 57\}    &      8 * 50 \\
\hline
\{8, 51, 59\}    &      8 * 54 \\
\hline
\{8, 53, 61\}    &      8 * 58 \\
\hline
\{8, 55, 63\}    &      8 * 62 \\
\hline
\{8, 57, 65\}    &      8 * 64 \\
\hline
\{8, 59, 67\}    &      8 * 64 \\
\hline
\{8, 61, 69\}    &      8 * 64 \\
\hline
\{8, 63, 71\}    &      8 * 64 \\
\hline
\{8, 65, 73\}    &      8 * 66 \\
\hline
\{8, 67, 75\}    &      8 * 70 \\
\hline
\{8, 69, 77\}    &      8 * 74 \\
\hline
\{8, 71, 79\}    &      8 * 78 \\
\hline
\{8, 73, 81\}    &      8 * 80 \\
\hline
\{8, 75, 83\}    &      8 * 80 \\
\hline
\{8, 77, 85\}    &      8 * 80 \\
\hline
\{8, 79, 87\}    &      8 * 80 \\
\hline
\{8, 81, 89\}    &      8 * 82 \\
\hline
\{8, 83, 91\}    &      8 * 86 \\
\hline
\{8, 85, 93\}    &      8 * 90 \\
\hline
\{8, 87, 95\}    &      8 * 94 \\
\hline
\{8, 89, 97\}    &      8 * 96 \\
\hline
\{8, 91, 99\}    &      8 * 96 \\
\hline
\{8, 93, 101\}    &      8 * 96 \\
\hline
\{8, 95, 103\}    &      8 * 96 \\
\hline
\{8, 97, 105\}    &      8 * 98 \\
\hline
\{8, 99, 107\}    &      8 * 102 \\
\hline
\{8, 101, 109\}    &      8 * 106 \\
\hline
\{8, 103, 111\}    &      8 * 110 \\
\hline

\end{tabular}
\end{tabular}

\end{document}